\documentclass[a4paper,11pt]{amsart}

\usepackage{graphicx}
\usepackage{mathptmx}
\usepackage{amsmath}
\usepackage{amssymb}
\usepackage{enumitem}
\usepackage{xcolor} 
\usepackage{pgfplots}

\newmuskip\pFqmuskip

\newcommand*\pFq[6][8]{%
  \begingroup 
  \pFqmuskip=#1mu\relax
  \mathcode`=\string"8000
  \begingroup\lccode`\~=`\,
  \lowercase{\endgroup\let~}\pFqcomma
  F^{#2}_{#3}{\left(\genfrac..{0pt}{}{#4}{#5}\bigg|#6\right)}%
  \endgroup
}
\newcommand{\pFqcomma}{\mskip\pFqmuskip}

\newtheorem{theorem}{Theorem}[section]
\newtheorem{lemma}[theorem]{Lemma}
\newtheorem{corollary}[theorem]{Corollary}
\newtheorem{proposition}[theorem]{Proposition}

\begin{document}

\title[Logarithms and Stirling numbers associated with delta series]{Logarithms and Stirling numbers associated with delta series}

\author{Dae San  Kim}
\address{Department of Mathematics, Sogang University, Seoul 121-742, Republic of Korea}
\email{dskim@sogang.ac.kr}

\author{Taekyun  Kim}
\address{Department of Mathematics, Kwangwoon University, Seoul 139-701, Republic of Korea}
\email{tkkim@kw.ac.kr}

\subjclass[2010]{11B68; 11B73; 11B83}
\keywords{logarithm associated with delta series; Stirling numbers of the second associated with delta series; Stirling numbers of the first kind associated with delta series; Schl\"omilch-type formula}

\begin{abstract}
This paper investigates the Stirling numbers of the first and second kind associated with a delta series $f(t)$, denoted by $S_{1}(n,k; f(t))$ and $S_{2}(n,k; f(t))$. These numbers provide a robust framework that satisfies the orthogonality and inverse relations, often lacking in recent probabilistic Stirling and $B$-Stirling numbers. Key contributions include the definition and analysis of the logarithm associated with a delta series $f(t)$. We further establish a Schl\"{o}milch-type formula, which provides an explicit connection between the two kinds of Stirling numbers. Using this formula, we derive another expression for the associated logarithm in terms of $S_{2}(n,k; f(t))$. Finally, we provide fifteen concrete examples to illustrate the versatility of this framework, demonstrating how it unifies and extends several known results in combinatorial analysis.
\end{abstract}

\maketitle

\section{Introduction}
The Bernoulli polynomials of order $\alpha$ and the Bernoulli numbers of order $\alpha$ are respectively defined by 
\begin{equation} \label{1}
\Big(\frac{t}{e^{t}-1}\Big)^{\alpha}e^{x t}=\sum_{n=0}^{\infty}B_{n}^{(\alpha)}(x)\frac{t^{n}}{n!},\quad B_{n}^{(\alpha)}=B_{n}^{(\alpha)}(0),
\end{equation}
for any complex number $\alpha$. \\
When $\alpha=1$, $B_{n}(x)=B_{n}^{1}(x)$ are called the Bernoulli polynomials. \par
The Stirling numbers of the second kind $S_{2}(n,k)$ are characterized by
\begin{equation} \label{2}
\frac{1}{k!}(e^{t}-1)^{k}=\sum_{n=k}^{\infty}S_{2}(n,k)\frac{t^{n}}{n!}, \quad x^{n}=\sum_{k=0}^{n}S_{2}(n,k)(x)_{k},
\end{equation}
while the Stirling numbers of the first kind are specified as
\begin{equation} \label{3}
\frac{1}{k!}\big(\log(1+t)\big)^{k}=\sum_{n=k}^{\infty}S_{1}(n,k)\frac{t^{n}}{n!},\quad (x)_{n}=\sum_{k=0}^{n}S_{1}(n,k)x^{k},
\end{equation}
where $(x)_{n}$ are the falling factorial sequence given by
\begin{equation} \label{4}
(x)_{0}=1,\quad (x)_{n}=x(x-1)\cdots(x-n+1),\quad(n \ge 1).
\end{equation} \par
Throughout this paper, let $\lambda$ denote any nonzero real number. 
The degenerate exponentials are introduced as
\begin{equation} \label{5}
\begin{aligned}
&e_{\lambda}^{x}(t)=(1+\lambda t)^{\frac{x}{\lambda}}=\sum_{n=0}^{\infty}(x)_{n,\lambda}\frac{t^{n}}{n!}, \\
&e_{\lambda}(t)=e_{\lambda}^{1}(t)=(1+\lambda t)^{\frac{1}{\lambda}}=\sum_{n=0}^{\infty}(1)_{n,\lambda}\frac{t^{n}}{n!},
\end{aligned} 
\end{equation}
where the degenerate falling factorial sequence is given by
\begin{equation} \label{6}
(x)_{0,\lambda}=1,\quad (x)_{n,\lambda}=x(x-\lambda)\cdots(x-(n-1)\lambda),\quad(n \ge 1).
\end{equation}
The degenerate Bernoulli polynomials of order $\alpha$ and the degenerate Bernoulli numbers of order $\alpha$ are respectively defined by
\begin{equation} \label{7}
\Big(\frac{t}{e_{\lambda}(t)-1}\Big)^{\alpha}e_{\lambda}^{x}(t)=\sum_{n=0}^{\infty}\beta_{n,\lambda}^{(\alpha)}(x)\frac{t^{n}}{n!},\quad \beta_{n,\lambda}^{(\alpha)}=\beta_{n,\lambda}^{(\alpha)}(0),
\end{equation}
for any complex number $\alpha$. \\
When $\alpha=1$, $\beta_{n, \lambda}^{(1)}(x)=\beta_{n, \lambda}(x)$ are called the degenerate Bernoulli polynomials. \par
The degenerate Stirling numbers of the second kind are specified as (see \cite{14,20})
\begin{equation} \label{8}
\begin{aligned}
&\frac{1}{k!}\big(e_{\lambda}(t)-1 \big)^{k}=\sum_{n=k}^{\infty}S_{2,\lambda}(n,k)\frac{t^{n}}{n!},\quad (k \ge 0), \\
&(x)_{n,\lambda}=\sum_{k=0}^{n}S_{2,\lambda}(n,k)(x)_{k},\quad(n \ge 0),
\end{aligned}
\end{equation}
while the degenerate Stirling numbers of the first kind are defined as (see \cite{12,20})
\begin{equation} \label{9}
\begin{aligned}
&\frac{1}{k!}\big(\log_{\lambda}(1+t)\big)^{k}=\sum_{n=k}^{\infty}S_{1,\lambda}(n,k)\frac{t^{n}}{n!}, \quad (k \ge 0),\\
&(x)_{n}=\sum_{k=0}^{n}S_{1,\lambda}(n,k)(x)_{k,\lambda},\quad (n \ge 0).
\end{aligned}
\end{equation}
Here $\log_{\lambda}(1+t)$ is called the degenerate logarithm, which is the compositional inverse of $e_{\lambda}(t)-1$, and given by
\begin{equation} \label{10}
\log_{\lambda}(1+t)=\frac{1}{\lambda}((1+t)^{\lambda}-1)=\sum_{n=1}^{\infty}(\lambda-1)_{n-1}\frac{t^{n}}{n!}.
\end{equation}
Taking the limits $\lambda \rightarrow 0$, we find that (see \eqref{5}, \eqref{6}, \eqref{7}, \eqref{8}, \eqref{9}, \eqref{10})
\begin{align*}
&(x)_{n,\lambda} \rightarrow x^{n}, \quad e_{\lambda}^{x}(t) \rightarrow e^{xt}, \quad e_{\lambda}(t) \rightarrow e^{t},\quad \log_{\lambda}(1+t) \rightarrow \log (1+t), \\
&\beta_{n,\lambda}^{(\alpha)}(x) \rightarrow B_{n}^{(\alpha)}(x), \quad S_{1,\lambda}(n,k) \rightarrow S_{1}(n,k), \quad 
S_{2,\lambda}(n,k) \rightarrow S_{2}(n,k).
\end{align*} \par
A formal power series $f(t)=\sum_{n=0}^{\infty}p_{n}\frac{t^{n}}{n!}$ is a delta series if $f(0)=0$ and $f^{\prime}(0) \ne 0$. If $f(t)$ is a delta series, then it has a compositional inverse $\bar{f}(t)$, which is also a delta series, satisfying $f(\bar{f}(t))=t=\bar{f}(f(t))$. Let the sequence of polynomials, $\left\{p_{n}(x) \right\}_{n=0}^{\infty}$, or $p_{n}(x)$ for short, be given by the generating function
\begin{equation} \label{11}
\sum_{n=0}^{\infty} p_{n}(x)\frac{t^{n}}{n!}=e^{x \bar{f}(t)}.
\end{equation}
Then $p_{n}(x)$ are called the sequence associated with $f(t)$, which is indicated by $p_{n}(x) \sim (1, f(t))$.
We denote the sequence of polynomials by $\bold{P}=\left\{p_{n}(x) \right\}_{n=0}^{\infty}$ or $\bold{P}=\left\{p_{n}(x) \right\}$, for brevity. We also let
\begin{equation} \label{12}
\sum_{n=0}^{\infty} p_{n}\frac{t^{n}}{n!}=e^{\bar{f}(t)},\quad p_{n}=p_{n}(1).
\end{equation} \par
We recall the following three Lagrange inversion formulas: for any formal power series $g(t)$ and any delta series $f(t)$, they are stated as (see \cite{5,25}):
\begin{flalign*}
&(A)\,\,\,[t^{n}]\,g(\bar{f}(t))=\frac{1}{n}[t^{n-1}]\,g^{\prime}(t)\Big(\frac{t}{f(t)}\Big)^{n}, \\
&(B)\,\,\,[t^{n}]\,\big(\bar{f}(t)\big)^{k}=\frac{k}{n}[t^{n-k}]\, \Big(\frac{t}{f(t)}\Big)^{n}, \\
&(C)\,\,\,[t^{n}]\,\bar{f}(t)=\frac{1}{n}[t^{n-1}]\,\Big(\frac{t}{f(t)}\Big)^{n}, &&
\end{flalign*}
where $\bar{f}(t)$ is the compositional inverse of $f(t)$, $n$ is any nonnegative integer and $k$ is any positive integer. \\
Applying the formula (C) to $f(t)=e_{\lambda}(t)-1$, we arrive at
\begin{equation} \label{13}
[t^{n}]\, \log_{\lambda}(1+t)=\frac{1}{n}[t^{n-1}]\,\Big(\frac{t}{e_{\lambda}(t)-1} \Big)^{n}=\frac{1}{n!}\beta_{n-1,\lambda}^{(n)}.
\end{equation}
We see from \eqref{9}, \eqref{10}, and \eqref{13} that
\begin{equation} \label{14}
\log_{\lambda}(1+t)=\sum_{n=1}^{\infty}S_{1,\lambda}(n,1)\frac{t^{n}}{n!}=\sum_{n=1}^{\infty}(\lambda-1)_{n-1}\frac{t^{n}}{n!}=\sum_{n=1}^{\infty}\beta_{n-1,\lambda}^{(n)}\frac{t^{n}}{n!}.
\end{equation}
In addition, applying the formula (B) to $f(t)=e_{\lambda}(t)-1$ gives
\begin{equation} \label{15}
[t^{n}]\,\big(\log_{\lambda}(1+t)\big)^{k}=\frac{k}{n}\beta_{n-k,\lambda}^{(n)}\frac{1}{(n-k)!}.
\end{equation}
Thus it follows from \eqref{15} that 
\begin{equation} \label{16}
\frac{1}{k!}\big(\log_{\lambda}(1+t)\big)^{k}=\sum_{n=k}^{\infty}\binom{n-1}{k-1}\beta_{n-k,\lambda}^{(n)}\frac{t^{n}}{n!},\quad (k \ge 0).
\end{equation}
Now, one observes from \eqref{9} and \eqref{16} that
\begin{equation} \label{17}
S_{1,\lambda}(n,k)=\binom{n-1}{k-1}\beta_{n-k,\lambda}^{(n)}, \quad (n \ge k).
\end{equation}
Taking the limits $\lambda \rightarrow 0$ in \eqref{14} and \eqref{17} yields
\begin{equation} \label{18}
S_{1}(n,1)=(-1)^{n-1}(n-1)!=B_{n-1}^{(n)},\quad S_{1}(n,k)=\binom{n-1}{k-1}B_{n-k}^{(n)}.
\end{equation} \par
The partial Bell polynomials are characterized by
\begin{equation*}
B_{n,k}(x_{1},x_{2},\dots,x_{n-k+1})=\sum_{\substack{j_{1}+j_{2}+\cdots=k \\ j_{1}+2j_{2}+\dots=n}}\frac{n!}{j_{1}!j_{2}!\cdots}\Big(\frac{x_{1}}{1!} \Big)^{j_{1}}\Big(\frac{x_{2}}{2!}\Big)^{j_{2}}\cdots,
\end{equation*}
where $j_{1}, j_{2}, \dots$ are nonnegative integers. The generating function of the partial Bell polynomials is specified as
\begin{equation} \label{19}
\frac{1}{k!}\bigg(\sum_{m=1}^{\infty}x_{m}\frac{t^{m}}{m!}\bigg)^{k}=\sum_{n=k}^{\infty}B_{n,k}(x_{1},x_{2},\dots,x_{n-k+1})\frac{t^{n}}{n!}, \quad (k \ge 0).
\end{equation} 
General references for this paper include \cite{5,6,23,24,25,26,27}. \par

Let $\mathbf{P}=\{p_{n}(x)\}_{n=0}^{\infty}$ be a sequence of polynomials such that $\deg p_{n}(x)=n$ and $p_{0}(x)=1$. In \cite{7}, the authors introduced Stirling numbers of the first and second kind associated with $\mathbf{P}$, defined respectively by the relations:
\begin{equation} \label{19-1}
(x)_{n} = \sum_{k=0}^{n}S_{1}(n,k; \mathbf{P})p_{k}(x), \quad
p_{n}(x) = \sum_{k=0}^{n}S_{2}(n,k; \mathbf{P})(x)_{k}.
\end{equation}
These numbers satisfy fundamental orthogonality and inverse relations, generalizing the classical Stirling numbers. \par
Let $Y$ be a random variable whose moment generating function exists in some neighborhood of the origin. Recently, Adell and Lekuona introduced the probabilistic Stirling numbers of the second kind associated with $Y$, denoted by $S_{2}^{Y}(n,k)$ \cite{3}. This was followed by Adell and B\'{e}nyi's definition of the probabilistic Stirling numbers of the first kind, $s_{Y}(n,k)$, based on the cumulant generating function \cite{1}. While these initial formulations provided a foundational framework, they lacked the standard orthogonality properties necessary for broader application in inversion formulas. In addition, the definition of $s_{Y}(n,k)$ was not designed to recover the classical Stirling numbers of the first kind, $S_{1}(n,k)$, in the deterministic case where $Y=1$. \par 
To extend the utility of these concepts, a refined version of the probabilistic Stirling numbers of the first kind, $S_{1}^{Y}(n,k)$, was proposed in \cite{8,28}. This updated definition ensures that $S_{1}^{Y}(n,k)$ and $S_{2}^{Y}(n,k)$ form a consistent pair satisfying the expected orthogonality and inverse relations, while also maintaining compatibility with the classical $S_{1}(n,k)$ when $Y=1$. Building on this refined approach, their degenerate counterparts--the probabilistic degenerate Stirling numbers of the second kind $S_{2,\lambda}^{Y}(n,k)$ and first kind $S_{1,\lambda}^{Y}(n,k)$--were established in \cite{8}. These variants similarly preserve the requisite orthogonality and inverse properties, successfully reducing to the degenerate Stirling numbers $S_{1,\lambda}(n,k)$ under the condition $Y=1$. Additionally, Adell and B\'{e}nyi \cite{2} generalized these to $B$-Stirling numbers associated with potential polynomials; however, these variants still lack the desired orthogonality properties. \par
In this paper, we investigate the Stirling numbers associated with a delta series $f(t)$ (i.e., $f(0)=0$ and $f^{\prime}(0) \ne 0$), denoted by $S_{2}(n,k; f(t))$ and $S_{1}(n,k; f(t))$, which unify probabilistic Stirling and $B$-Stirling numbers. They correspond to the Stirling numbers associated with polynomial sequence $\mathbf{P}= \{p_{n}(x)\}_{n=0}^{\infty}$, defined by the generating function $e^{x \bar{f}(t)}=\sum_{n=0}^{\infty}p_{n}(x)\frac{t^{n}}{n!}$, where $\bar{f}(t)$ is the compositional inverse of $f(t)$ (see \eqref{19-1}). They also correspond to the $B$-Stirling numbers with `B' given by $B(t)=e^{\bar{f}(t)}$ (see \cite{2}). Furthermore, the probabilistic Stirling numbers are the Stirling numbers associated with the delta series $f(t)$, whose compositional inverse is given by $\bar{f}(t)=\log E[e_{\lambda}^{Y}(t)]$ (see Example (a) in Section 3).
These Stirling numbers satisfy orthogonality and inverse relations. As $\mathbf{P}= \{p_{n}(x)\}_{n=0}^{\infty}$ determines $f(t)$ and vice versa, $S_{2}(n,k; f(t))$ and $S_{1}(n,k; f(t))$ are also denoted respectively by $S_{2}(n,k; \mathbf{P})$ and $S_{1}(n,k; \mathbf{P})$. This correspondence extends to all related quantities whenever $\mathbf{P}$ and $f(t)$ are related by the generating function $e^{x \bar{f}(t)}=\sum_{n=0}^{\infty}p_{n}(x)\frac{t^{n}}{n!}$.
We define the logarithm associated with $f(t)$, $\log_{f(t)}(1+t)$, which is given by
\begin{equation*}
\log_{f(t)}(1+t)=f(\log(1+t))=\sum_{n=1}^{\infty}S_{1}(n,1; f(t))\frac{t^{n}}{n!}=\sum_{n=1}^{\infty}B_{n-1,\bar{f}(t)}^{(n)}\frac{t^{n}}{n!},
\end{equation*}
where $B_{n,f(t)}^{(\alpha)}(x)$ denotes the Bernoulli polynomials of order $\alpha$ associated with $f(t)$, defined by
\begin{equation*}
\left(\frac{t}{e^{f(t)}-1}\right)^{\alpha}e^{xf(t)}=\sum_{n=0}^{\infty}B_{n,f(t)}^{(\alpha)}(x)\frac{t^{n}}{n!}.
\end{equation*}
It is worth noting that the generating function of $S_{1}(n,k; f(t))$ takes the form
\begin{equation*}
\frac{1}{k!}\big(\log_{f(t)}(1+t)\big)^{k}=\sum_{n=k}^{\infty}S_{1}(n,k; f(t))\frac{t^{n}}{n!}.
\end{equation*}
We derive a Schl\"{o}milch-type formula for these Stirling numbers. In turn, this gives an alternative expression for the logarithm associated with $f(t)$:
\begin{equation*}
\log_{f(t)}(1+t)=\sum_{n=1}^{\infty}\sum_{j=0}^{n-1}\binom{2n-1}{n-1-j}(-1)^{j}\left(\bar{f}^{\prime}(0)\right)^{-n-j}S_{2}(n-1+j,j; f(t))\frac{t^{n}}{n!}.
\end{equation*}
In Section 3, we provide fifteen concrete examples to illustrate our results, drawing primarily from \cite{7} and \cite{9}. \par

\section{Logarithms and Stirling numbers associated with delta series}
Let $f(t)$ be a delta series. So $f(0)=0$, and  $f^{\prime}(0) \ne 0$. Here we introduce the logarithm $\log_{f(t)}(1+t)$, called the logarithm associated with $f(t)$. It is also denoted by $\log_{\mathbf{P}}(1+t)$ and called the logarithm associated with $\mathbf{P}$. Here $\mathbf{P}=\left\{p_{n}(x)\right\}_{n=0}^{\infty}$, with $p_{n}(x) \sim (1, f(t))$. Thus we have (see \eqref{11})
\begin{equation*}
\sum_{n}^{\infty}p_{n}(x)\frac{t^{n}}{n!}=e^{x\bar{f}(t)}.
\end{equation*}
We also recall that (see \eqref{12})
\begin{equation} \label{20}
\sum_{n}^{\infty}p_{n}\frac{t^{n}}{n!}=e^{\bar{f}(t)},\quad \mathrm{with}\,\, p_{n}=p_{n}(1).
\end{equation}
Note here that the coefficient of the linear term of \eqref{20} is given by $p_{1}=\bar{f}^{\prime}(0)$. \par
We define the {\it{Bernoulli polynomials of order $\alpha$ associated with $f(t)$ or $\mathbf{P}$}} by
\begin{equation} \label{21}
\Big(\frac{t}{e^{f(t)}-1}\Big)^{\alpha}e^{xf(t)}=\sum_{n=0}^{\infty}B_{n,f(t)}^{(\alpha)}(x)\frac{t^{n}}{n!},
\end{equation}
for any complex number $\alpha$. \\
If $f(t)=t$, then $B_{n,f(t)}^{(\alpha)}(x)=B_{n}^{(\alpha)}(x)$ are the classical Bernoulli polynomials of order $\alpha$ (see \eqref{1}); if $f(t)=\frac{1}{\lambda}\log(1+\lambda t)$, then $B_{n,f(t)}^{(\alpha)}(x)=\beta_{n}^{(\alpha)}(x)$ are the degenerate Bernoulli polynomials of order $\alpha$ (see \eqref{7}). \\
When $\alpha=1$, they are called the {\it{Bernoulli polynomials associated with $f(t)$ or $\mathbf{P}$}} and given by
\begin{equation*}
\frac{t}{e^{f(t)}-1}e^{xf(t)}=\sum_{n=0}^{\infty}B_{n,f(t)}(x)\frac{t^{n}}{n!}.
\end{equation*}

The Stirling numbers of the second kind associated with $f(t)$ or $\mathbf{P}$, denoted by either $S_{2}(n,k; f(t))$ or $S_{2}(n,k; \mathbf{P})$, are defined by
\begin{equation} \label{22}
\frac{1}{k!}\big(e^{\bar{f}(t)}-1\big)^{k}=\sum_{n=k}^{\infty}S_{2}(n,k; f(t))\frac{t^{n}}{n!}=\sum_{n=k}^{\infty}S_{2}(n,k; \mathbf{P})\frac{t^{n}}{n!}.
\end{equation}
If $f(t)=t$, then $S_{2}(n,k; f(t))=S_{2}(n,k)$ are the ordinary Stirling numbers of the second (see \eqref{2}); if $f(t)=\frac{1}{\lambda}(e^{\lambda t}-1)$, then $S_{2}(n,k; f(t))=S_{2,\lambda}(n,k)$ are the degenerate Stirling numbers of the second kind (see \eqref{8}). \\
The Bell polynomials associated with $f(t)$, denoted by $\mathrm{Bel}_{n,f(t)}(x)$, are defined by
\begin{equation*}
e^{x (e^{\bar{f}(t)}-1)}=\sum_{n=0}^{\infty}\mathrm{Bel}_{n, f(t)}(x)\frac{t^{n}}{n!}, \quad \mathrm{Bel}_{n, f(t)}(x)=\sum_{n=0}^{n}S_{2}(n,k; f(t))x^{k}.
\end{equation*} \par
The Stirling numbers of the first kind associated with $f(t)$ or $\mathbf{P}$, denoted by either $S_{1}(n,k; f(t))$ or $S_{1}(n,k; \mathbf{P})$, are defined by
\begin{equation} \label{23}
\frac{1}{k!}\big(\bar{e}_{f(t)}(t)\big)^{k}=\sum_{n=k}^{\infty}S_{1}(n,k; f(t))\frac{t^{n}}{n!}=\sum_{n=k}^{\infty}S_{1}(n,k; \mathbf{P})\frac{t^{n}}{n!},
\end{equation}
where $\bar{e}_{f(t)}(t)$ is the compositional inverse of $e_{f(t)}(t)=e^{\bar{f}(t)}-1$. Here we note that $e_{f(t)}(t)$ is a delta series, since $e_{f(t)}(0)=0$, $e_{f(t)}^{\prime}(0)=\bar{f}^{\prime}(0) \ne 0$. \\
If $f(t)=t$, then $S_{1}(n,k; f(t))=S_{1}(n,k)$ are the ordinary Stirling numbers of the first kind (see \eqref{3}); if $f(t)=\frac{1}{\lambda}(e^{\lambda t}-1)$, then $S_{1}(n,k; f(t))=S_{1,\lambda}(n,k)$ are the degenerate Stirling numbers of the first kind (see \eqref{9}). \par
It is shown in \cite{7} that, for $p_{n}(x) \sim (1, f(t))$, they are equivalently characterized by (see \eqref{4})\\
\begin{align}
&p_{n}(x)=\sum_{k=0}^{n}S_{2}(n,k; f(t))(x)_{k}, \label{24} \\
&(x)_{n}=\sum_{k=0}^{n}S_{1}(n,k; f(t))p_{n}(x). \label{25}
\end{align}
Among other things, from \eqref{24} and \eqref{25}, we note that $S_{1}(n,k; f(t))$ and $S_{2}(n,k; f(t))$ satisfy the orthogonality and inverse relations.
\begin{proposition}
The following orthogonality and inverse relations are valid for $S_{1}(n,k; f(t))$ and $S_{2}(n,k; f(t))$.
\begin{flalign*}
&(a)\,\, \,\sum_{k=l}^{n} S_{2}(n,k; f(t))S_{1}(k,l; f(t))=\delta_{n,l}, \quad \sum_{k=l}^{n} S_{1}(n,k; f(t))S_{2}(k,l; f(t))=\delta_{n,l}, \\
&(b)\,\, a_{n}=\sum_{k=0}^{n}S_{2}(n,k; f(t)) b_{k}\,\, \iff \,\, b_{n}=\sum_{k=0}^{n}S_{1}(n,k; f(t))a_{k}, \\ 
&(c)\,\, a_{n}=\sum_{k=n}^{m}S_{2}(k,n; f(t))b_{k} \,\, \iff \,\, b_{n}=\sum_{k=n}^{m}S_{1}(k,n; f(t))a_{k}. &&
\end{flalign*}
\end{proposition}
By using the formula (C) and \eqref{21}, we derive that
\begin{equation} \label{26}
\begin{aligned}
[t^{n}]\,\bar{e}_{f(t)}(t)&=\frac{1}{n}[t^{n-1}]\,\bigg(\frac{t}{e^{\bar{f}(t)}-1}\bigg)^{n} \\
&=\frac{1}{n}[t^{n-1}]\,\sum_{k=0}^{\infty}B_{k,\bar{f}(t)}^{(n)}\frac{t^{k}}{k!}=\frac{1}{n!} B_{n-1,\bar{f}(t)}^{(n)}.
\end{aligned}
\end{equation}
Consequently, from \eqref{26} we have
\begin{equation} \label{27}
\bar{e}_{f(t)}(t)=\sum_{n=1}^{\infty}B_{n-1,\bar{f}(t)}^{(n)}\frac{t^{n}}{n!}
\end{equation}
Also, by using the formula (B), we find that
\begin{equation} \label{28}
\begin{aligned}
[t^{n}]\,\big(\bar{e}_{f(t)}(t)\big)^{k}&=\frac{k}{n}[t^{n-k}]\,\bigg(\frac{t}{e^{\bar{f}(t)}-1}\bigg)^{n}\\
&=\frac{k}{n}[t^{n-k}]\,\sum_{k=0}^{\infty}B_{k,\bar{f}(t)}^{(n)}\frac{t^{k}}{k!}=
\frac{k}{n}\frac{1}{(n-k)!}B_{n-k,\bar{f}(t)}^{(n)}.
\end{aligned}
\end{equation}
Hence from \eqref{28} we arrive at
\begin{equation} \label{29}
\frac{1}{k!}\big(\bar{e}_{f(t)}(t)\big)^{k}=\sum_{n=k}^{\infty}\binom{n-1}{k-1}B_{n-k,\bar{f}(t)}^{(n)}\frac{t^{n}}{n!}.
\end{equation} \par
Now, we define the {\it{logarithm associated with the delta seires $f(t)$ or the associated sequence $\mathbf{P}$}} by
\begin{equation} \label{30}
\log_{f(t)}(1+t)=\log_{\mathbf{P}}(1+t)= f(\log(1+t)).
\end{equation}
Observe that $f(\log(1+t))$ is the compositional inverse of $e_{f(t)}(t)=e^{\bar{f}(t)}-1$.
Thus we see from \eqref{23}, \eqref{27}, and \eqref{30} that 
\begin{equation} \label{31}
\log_{f(t)}(1+t)=f(\log(1+t)=\bar{e}_{f(t)}(t)=\sum_{n=1}^{\infty}S_{1}(n,1; f(t))\frac{t^{n}}{n!}=\sum_{n=1}^{\infty}B_{n-1,\bar{f}(t)}^{(n)}\frac{t^{n}}{n!}.
\end{equation}
From \eqref{23} and \eqref{31}, we note that
\begin{equation} \label{32}
\begin{aligned}
\frac{1}{k!}\big(\log_{f(t)}(1+t)\big)^{k}&=\sum_{n=k}^{\infty}S_{1}(n,k; f(t))\frac{t^{n}}{n!}\\
&=\frac{1}{k!}\bigg(\sum_{m=1}^{\infty}S_{1}(m,1; f(t))\frac{t^{m}}{m!}\bigg)^{k}\\
&=\frac{1}{k!}\bigg(\sum_{m=1}^{\infty}B_{m-1,\bar{f}(t)}^{(m)} \frac{t^{m}}{m!}\bigg)^{k}.
\end{aligned}
\end{equation}
Then, from \eqref{19}, \eqref{23}, \eqref{29}, and \eqref{32}, we obtain the following theorem.
\begin{theorem}
For any integers $n \ge k \ge 0$, we have
\begin{align*}
S_{1}(n,k; f(t))&=\binom{n-1}{k-1}B_{n-k,\bar{f}(t)}^{(n)} \\
&=B_{n,k}\big(S_{1}(1,1; f(t)),\,S_{1}(2,1; f(t)),\,\dots,\,S_{1}(n-k+1,1; f(t))\big) \\
&=B_{n,k}\Big(B_{0,\bar{f}(t)}^{(1)},\,B_{1,\bar{f}(t)}^{(2)},\,\dots\,B_{n-k,\bar{f}(t)}^{(n-k+1)}\Big).
\end{align*}
\end{theorem}
To derive yet another expression for the logarithm associated with $f(t)$, we need to prove a Schl\"omilch-type formula for the Stirling numbers associated with $f(t)$.
The next two lemmas and Theorem 2.7 can be proved just as those in Lemma 2.1, Lemma 2.2 and Theorem 2.1 of \cite{28}. However, we give brief sketches of the proofs. The details are left to the reader.
\begin{lemma}
Let $f(t)$ be a delta series, with $e^{\bar{f}(t)}=\sum_{n=0}^{\infty}p_{n}\frac{t^{n}}{n!}$. For any integers $n \ge k \ge 0$, we have
\begin{align*}
B_{n,k}&\Big(\frac{p_{2}}{2},\frac{p_{3}}{3},\dots,\frac{p_{n-k+2}}{n-k+2}\Big) \\
&=\sum_{j=0}^{k}\binom{n+k}{k-j}\frac{n!}{(n+k)!}\big(-p_{1})^{k-j}S_{2}(n+j,j; f(t)).
\end{align*}
\begin{proof}
On the one hand, we observe that
\begin{align*}
\frac{1}{k!}\Big(e^{\bar{f}(t)}-1-p_{1}t\Big)^{k}&=\frac{1}{k!}t^{k}\Big(\sum_{n=1}^{\infty}\frac{p_{n+1}}{n+1}\frac{t^{n}}{n!}\Big)^{k} \\
&=\sum_{n=k}^{\infty}\frac{t^{n+k}}{n!}B_{n,k}\Big(\frac{p_{2}}{2},\frac{p_{3}}{3},\dots,\frac{p_{n-k+2}}{n-k+2}\Big).
\end{align*}
On the other hand, we also find that
\begin{align*}
\frac{1}{k!}\Big(e^{\bar{f}(t)}-1-p_{1}t\Big)^{k}&=\sum_{j=0}^{k}\frac{1}{(k-j)!}(-p_{1}t)^{k-j}\frac{1}{j!}\big(e^{\bar{f}(t)}-1\big)^{j} \\
&=\sum_{j=0}^{k}\frac{1}{(k-j)!}(-p_{1}t)^{k-j}\sum_{n=0}^{\infty}S_{2}(n+j,j; f(t))\frac{t^{n+j}}{(n+j)!} \\
&=\sum_{n=0}^{\infty}\frac{t^{n+k}}{n!}\sum_{j=0}^{k}\binom{n+k}{k-j}\frac{n!}{(n+k)!}(-p_{1})^{k-j}S_{2}(n+j,j; f(t)),
\end{align*}
where we used $\frac{n!}{(k-j)!(n+j)!}=\binom{n+k}{k-j}\frac{n!}{(n+k)!}$.
\end{proof}
\end{lemma}

\begin{lemma}
Let $f(t)$ be a delta series, with $e^{\bar{f}(t)}=\sum_{n=0}^{\infty}p_{n}\frac{t^{n}}{n!}$. For any integer $n \ge 0$, we have
\begin{align*}
B_{n,\bar{f}(t)}^{(\alpha)}=\sum_{k=0}^{n}(-\alpha)_{k}p_{1}^{-\alpha-k}B_{n,k}\Big(\frac{p_{2}}{2},\frac{p_{3}}{3},\dots,\frac{p_{n-k+2}}{n-k+2}\Big). \\
\end{align*}
\begin{proof}
We proceed as follows:
\begin{align*}
\sum_{n=0}^{\infty}B_{n,\bar{f}(t)}^{(\alpha)}\frac{t^{n}}{n!}&=\Big(\frac{e^{\bar{f}(t)}-1}{t}\Big)^{-\alpha}=p_{1}^{-\alpha}\Big(1+\sum_{n=1}^{\infty}\frac{p_{n+1}}{p_1}\frac{t^{n}}{(n+1)!}\Big)^{-\alpha} \\
&=\sum_{k=0}^{\infty}(-\alpha)_{k}p_{1}^{-\alpha-k}\frac{1}{k!}\Big(\sum_{n=1}^{\infty}\frac{p_{n+1}}{n+1}\frac{t^{n}}{n!}\Big)^{k} \\
&=\sum_{n=0}^{\infty}\frac{t^{n}}{n!}\sum_{k=0}^{n}(-\alpha)_{k}p_{1}^{-\alpha-k}B_{n,k}\Big(\frac{p_{2}}{2},\frac{p_{3}}{3},\dots,\frac{p_{n-k+2}}{n-k+2}\Big).
\end{align*}
\end{proof}
\end{lemma}
Combining the above two lemmas, recalling $p_{1}=\bar{f}^{\prime}(0)$ and noting $\frac{n!k!}{(n+k)!}\binom{n+k}{k-j}=\binom{k}{j}\binom{n+j}{j}^{-1}$ , we arrive at the next result.
\begin{theorem}
Let $f(t)$ be a delta series. For any integer $n \ge 0$, the following identity holds true.
\begin{equation*}
B_{n,\bar{f}(t)}^{(\alpha)}=\sum_{k=0}^{n}\sum_{j=0}^{k}\binom{\alpha+k-1}{k}\binom{k}{j}\binom{n+j}{j}^{-1}(-1)^{j}\big(\bar{f}^{\prime}(0)\big)^{-\alpha-j}S_{2}(n+j,j; f(t)).
\end{equation*}
\end{theorem}
For $\alpha=1$, we obtain the following expression from Theorem 2.5.
\begin{corollary}
Let $f(t)$ be a delta series. For any integer $n \ge 0$, the following identity holds true.
\begin{equation*}
B_{n,\bar{f}(t)}=\sum_{j=0}^{n}\binom{n+1}{j+1}\binom{n+j}{j}^{-1}(-1)^{j}\big(\bar{f}^{\prime}(0)\big)^{-1-j}S_{2}(n+j,j; f(t)).
\end{equation*}
\end{corollary}
Again, the following Schl\"omilch-type formula for the Stirling numbers associated with $f(t)$ is obtained by proceeding just as in the proof of Theorem 3.1 of \cite{28}. Nevertheless, we give the short proof. For this, we need the following identities:
\begin{equation} \label{33}
\begin{aligned}
&\sum_{i=j}^{n-k}\binom{n+i-1}{i}\binom{i}{j}=\frac{n}{n+j}\binom{2n-k}{n}\binom{n-k}{j}, \\
&\binom{n-1}{k-1}\binom{2n-k}{n}\frac{n}{n+j}\binom{n-k}{j}\binom{n-k+j}{j}^{-1}=\binom{n+j-1}{n+j-k}\binom{2n-k}{n-k-j}.
\end{aligned}
\end{equation}
The second identity in \eqref{33} is straightforward. The first identity in \eqref{33} follows by using $\sum_{i=j}^{m}\binom{i}{j}=\binom{m+1}{j+1}$ and noting that
\begin{equation*}
\sum_{i=j}^{n-k}\frac{n+j}{n}\binom{n+i-1}{i}\binom{i}{j}=\binom{n+j}{j}\sum_{i=n+j-1}^{2n-k-1}\binom{i}{n+j-1}.
\end{equation*}
\begin{theorem}
Let $f(t)$ be a delta series. For any integers $n \ge k \ge 0$, we have
\begin{equation*}
S_{1}(n,k; f(t))=\sum_{j=0}^{n-k}\binom{n+j-1}{n+j-k}\binom{2n-k}{n-k-j}(-1)^{j}\big(\bar{f}^{\prime}(0)\big)^{-n-j}S_{2}(n-k+j,j; f(t)).
\end{equation*}
\end{theorem}
\begin{proof}
From Theorem 2.5, we see that 
\begin{equation} \label{34}
\begin{aligned}
&B_{n-k,\bar{f}(t)}^{(n)}\\
&=\sum_{i=0}^{n-k}\sum_{j=0}^{i}\binom{n+i-1}{i}\binom{i}{j}\binom{n-k+j}{j}^{-1}(-1)^{j}\big(\bar{f}^{\prime}(0)\big)^{-n-j}S_{2}(n-k+j,j; f(t)) \\
&=\sum_{j=0}^{n-k}\binom{n-k+j}{j}^{-1}(-1)^{j}\big(\bar{f}^{\prime}(0)\big)^{-n-j}S_{2}(n-k+j,j; f(t))\sum_{i=j}^{n-k}\binom{n+i-1}{i}\binom{i}{j}.
\end{aligned}
\end{equation}
From Theorem 2.2, \eqref{33}, and \eqref{34}, we note
\begin{align*}
&S_{1}(n,k; f(t))=\sum_{j=0}^{n-k}\binom{n-1}{k-1}\binom{2n-k}{n}\frac{n}{n+j}\binom{n-k}{j}\binom{n-k+j}{j}^{-1}\\
&\quad\quad\quad\quad\quad\quad \times (-1)^{j}\big(\bar{f}^{\prime}(0)\big)^{-n-j}S_{2}(n-k+j,j; f(t)) \\
&=\sum_{j=0}^{n-k} \binom{n+j-1}{n+j-k}\binom{2n-k}{n-k-j}(-1)^{j}\big(\bar{f}^{\prime}(0)\big)^{-n-j}S_{2}(n-k+j,j; f(t)).
\end{align*}
\end{proof}

From \eqref{31} and Theorem 2.7, we obtain another expression for the logarithm in terms of the Stirling numbers of the second kind associated with $f(t)$.
\begin{theorem}
Let $f(t)$ be a delta series. Then we have
\begin{equation*}
\log_{f(t)}(1+t)=\sum_{n=1}^{\infty}\sum_{j=0}^{n-1}\binom{2n-1}{n-1-j}(-1)^{j}\big(\bar{f}^{\prime}(0)\big)^{-n-j}S_{2}(n-1+j,j; f(t))\frac{t^{n}}{n!}.
\end{equation*}
\end{theorem}

\section{Examples}
In (a)-(o) below, we provide explicit expressions for the following quantities (see \eqref{22}, \eqref{23}, \eqref{31}):
\begin{align*}
&S_{2}(n,k; f(t))=S_{2}(n,k; \mathbf{P}),\quad S_{1}(n,k; f(t))=S_{1}(n,k; \mathbf{P}), \\ &\log_{f(t)}(1+t)=\log_{\mathbf{P}}(1+t),
\end{align*}
where $\mathbf{P}=\{p_{n}(x)\}_{n=0}^{\infty}$\,, with $p_{n}(x) \sim (1, f(t))$.

\noindent(a) In recent years, probabilistic versions of Stirling numbers of both kinds have been explored by several authors (see \cite{1,3,8,21,28}). We assume that $Y$ is a random variable such that the moment-generating function of $Y$
\begin{equation*} 
E[e^{tY}]=\sum_{n=0}^{\infty}E[Y^{n}]\frac{t^{n}}{n!}, \quad (|t|<r)
\end{equation*}
exists for some $r > 0$, where $E$ is the mathematical expectation (see \cite{1,3,26}). We assume further that $E[Y] \ne 0$. \par
We recall that the probabilisitc degenerate Stirling numbers of the second kind associated with $Y$, $S_{2,\lambda}^{Y}(n,k)$, the probabilisitc degenerate Stirling numbers of the first kind associated with $Y$, $S_{1,\lambda}^{Y}(n,k)$, the probabilistic degenerate logarithm associated with $Y$, $\log_{\lambda}^{Y}(1+t)$, and the probabilistic degenerate Bernoulli polynomials of order $\alpha$ associated with $Y$, $\beta_{n,\lambda}^{(Y,\alpha)}(x)$, are respectively given by (see \cite{9,15,22})
\begin{equation} \label{35}
\begin{aligned}
&\frac{1}{k!}(E[e_{\lambda}^{Y}(t)]-1)^{k}=\sum_{n=k}^{\infty}S_{2,\lambda}^{Y}(n,k)\frac{t^{n}}{n!},\quad (k \ge 0), \\
&\frac{1}{k!}\big(\bar{e}_{Y,\lambda}(t)\big)^{k}=\sum_{n=k}^{\infty}S_{1,\lambda}^{Y}(n,k)\frac{t^{n}}{n!}, \quad (k \ge 0), \\
&\log_{\lambda}^{Y}(1+t)=\bar{e}_{Y,\lambda}(t)=\sum_{n=1}^{\infty}S_{1,\lambda}^{Y}(n,1)\frac{t^{n}}{n!}, \\
&\bigg(\frac{t}{E[e_{\lambda}^{Y}(t)]-1}\bigg)^{\alpha}\Big(E[e_{\lambda}^{Y}(t)]\Big)^{x}=\sum_{n=0}^{\infty}\beta_{n,\lambda}^{(Y,\alpha)}(x)\frac{t^{n}}{n!},
\end{aligned}
\end{equation}
where $e_{Y,\lambda}(t)=E[e_{\lambda}^{Y}(t)]-1=E[Y]t+\sum_{n=2}^{\infty}E[(Y)_{n,\lambda}]\frac{t^{n}}{n!}$ is a delta series.
In addition, we observe that
\begin{align*}
\log E[e_{\lambda}^{Y}(t)]&=\log(1+E[e_{\lambda}^{Y}(t)]-1) \\
&=\sum_{k=1}^{\infty}(-1)^{k-1}(k-1)!\frac{1}{k!}\big(E[e_{\lambda}^{Y}(t)]-1\big)^{k} \\
&=\sum_{k=1}^{\infty}(-1)^{k-1}(k-1)!\sum_{n=k}^{\infty}S_{2,\lambda}^{Y}(n,k)\frac{t^{n}}{n!} \\
&=E[Y]t+\sum_{n=2}^{\infty}\sum_{k=1}^{n}(-1)^{k-1}(k-1)!S_{2,\lambda}^{Y}(n,k)\frac{t^{n}}{n!}.
\end{align*}
Thus $\log E[e_{\lambda}^{Y}(t)]$ is a delta series. Now, with $\bar{f}(t)=\log E[e_{\lambda}^{Y}(t)]$, we see that
\begin{equation} \label{36}
e^{\bar{f}(t)}-1=E[e_{\lambda}^{Y}(t)]-1,\quad \bar{e}_{Y,\lambda}(t)=\bar{e}_{f(t)}(t).
\end{equation}
Consequently, we see from \eqref{35} and \eqref{36} that
\begin{align*}
&S_{2}(n,k; f(t))=S_{2,\lambda}^{Y}(n,k), \quad S_{1}(n,k; f(t))=S_{1,\lambda}^{Y}(n,k),\\
&log_{f(t)}(1+t)=\log_{\lambda}^{Y}(1+t), \quad B_{n,\bar{f}(t)}^{(\alpha)}(x)=\beta_{n,\lambda}^{(Y,\alpha)}(x).
\end{align*} \par
In \cite{9}, some explicit expressions of $S_{2,\lambda}^{Y}(n,k),\,\, S_{1,\lambda}^{Y}(n,k)$,\, and $\log_{\lambda}^{Y}(1+t)$ are given for several discrete and continuous random variables $Y$. For example, it is shown that, for the uniform random variable $(Y \sim U[0,1])$ (see \cite{26}),
\begin{align*}
&\log_{\lambda}^{Y}(1+t)=\sum_{n=1}^{\infty}S_{1,\lambda}^{Y}(n,1)\frac{t^{n}}{n!}=\sum_{n=1}^{\infty}2^{n}\sum_{m=0}^{n-1}\binom{n-1}{m} \\
&\times \sum_{j_1+j_2+\cdots+j_n=m}\binom{m}{j_1,j_2,\dots,j_n}A_{2,j_1}A_{2,j_2}\cdots A_{2,j_n}\lambda^{n-m-1}\frac{t^n}{n!},
\end{align*}
where the numers $A_{2,n}$'s are given by (see \cite{10})
\begin{equation*}
\frac{\frac{1}{2!}t^{2}}{e^{t}-1-t}=\sum_{n=0}^{\infty}A_{2,n}\frac{t^{n}}{n!}.
\end{equation*}
(b) Let $\mathbf{P}=\left\{x^{n}\right\}$. Then $x^{n} \sim (1,t)$ (see \eqref{18},\,\cite{5,24,25}). It is evident that 
\begin{align*}
&S_{2}(n,k;\mathbf{P})=S_{2}(n,k),\quad S_{1}(n,k;\mathbf{P})=S_{1}(n,k)=\binom{n-1}{k-1}B_{n-k}^{(n)}, \\
& \log_{\mathbf{P}}(1+t)=\log(1+t)=\sum_{n=1}^{\infty}S_{1}(n,1)\frac{t^n}{n!}=\sum_{n=1}^{\infty}B_{n-1}^{(n)}\frac{t^{n}}{n!}.
\end{align*}
(c) Let $\mathbf{P}=\left\{(x)_{n,\lambda}\right\}$ be the sequence of degenerate falling factorials (see \eqref{6}, \eqref{14}, \eqref{17}). Then $(x)_{n,\lambda} \sim \big(1, f(t)=\frac{1}{\lambda}(e^{\lambda t}-1)\big)$ (see \cite{17,19}). It is easy to see that
\begin{align*}
&S_{2}(n,k;\mathbf{P})=S_{2,\lambda}(n,k),\quad S_{1}(n,k;\mathbf{P})=S_{1,\lambda}(n,k)=\binom{n-1}{k-1}\beta_{n-k,\lambda}^{(n)},\\
&\log_{\mathbf{P}}(1+t)=\log_{\lambda}(1+t)=\sum_{n=1}^{\infty}S_{1,\lambda}(n,1)\frac{t^{n}}{n!}=\sum_{n=1}^{\infty}\beta_{n-1,\lambda}^{(n)}\frac{t^{n}}{n!}.
\end{align*}
(d) Let $\mathbf{P}=\left\{\langle{x \rangle}_{n}\right\}$ be the sequence of rising factorials which are given by (see \cite{5,11,24,25})
\begin{equation*}
\langle x \rangle_{0}=1,\quad \langle x \rangle_{n}=x(x+1)\cdots(x+n-1),\quad(n \ge 1).
\end{equation*}
Then $\langle{x \rangle}_{n} \sim (1, 1-e^{-t})$. We show that
\begin{align*}
&S_{2}(n,k;\mathbf{P})=L(n,k), \quad S_{1}(n,k;\mathbf{P})=(-1)^{n-k}L(n,k), \\
&\log_{\mathbf{P}}(1+t)=\sum_{n=1}^{\infty}(-1)^{n-1}L(n,1)\frac{t^{n}}{n!}=\frac{t}{1+t},
\end{align*}
where the (unsigned) Lah numbers are specified as
\begin{equation*}
\frac{1}{k!}\Big(\frac{t}{1-t}\Big)^{k}=\sum_{n=k}^{\infty}L(n,k)\frac{t^{n}}{n!}, \quad (k \ge 0),\quad L(n,k)=\frac{n!}{k!}\binom{n-1}{k-1},\quad (n \ge k).
\end{equation*}
(e) Let $\mathbf{P}=\left\{\langle x \rangle_{n,\lambda}\right\}$ be the sequence of degenerate rising factorials which are given by (see \cite{13})
\begin{equation*}
\langle x \rangle_{0,\lambda}=1,\quad \langle x \rangle_{n,\lambda}=x(x+\lambda)\cdots(x+(n-1)\lambda),\quad(n \ge 1).
\end{equation*}
Then $\langle{x \rangle}_{n,\lambda} \sim \big(1, \frac{1}{\lambda}(1-e^{-\lambda t})\big)$. Let $L_{\lambda}(n,k)$ be the degenerate Lah numbers given by
\begin{align*}
\sum_{n=k}^{\infty}L_{\lambda}(n,k)\frac{t^{n}}{n!}&=\frac{1}{k!}\big(e^{-\frac{1}{\lambda}\log (1-\lambda t)}-1\big)^{k}=\frac{1}{k!}\big(e_{-\lambda}(t)-1\big)^{k}.
\end{align*}
We derive that
\begin{align*}
&S_{2}(n,k;\mathbf{P})=L_{\lambda}(n,k), \quad S_{1}(n,k;\mathbf{P})=(-\lambda)^{n-k}L_{\frac{1}{\lambda}}(n,k)=S_{1,-\lambda}(n,k), \\
& \log_{\mathbf{P}}(1+t)=\sum_{n=1}^{\infty}S_{1,-\lambda}(n,1)\frac{t^{n}}{n!}=\log_{-\lambda}(1+t).
\end{align*}
(f) The central factorials $x^{[n]}$ are defined by (see \cite{4,5,25})
\begin{align*}
x^{[n]}=x\big(x+\frac{1}{2}n-1\big)_{n-1}, \,\,(n \ge 1),\quad x^{[0]}=1,
\end{align*}
where $x^{[n]} \sim (1, f(t)=e^{\frac{t}{2}}-e^{-\frac{t}{2}})$. Here we note that
\begin{equation*}
\bar{f}(t)=2\log\bigg(\frac{t+\sqrt{t^{2}+4}}{2}\bigg).
\end{equation*}
Carlitz-Riordan and Riordan (see \cite{3,4,11,23}) discussed the numbers $T_{1}(n,k)$ and $T_{2}(n,k)$ which are respectively called the central factorial numbers of the first kind and the second kind, and defined by
\begin{align*}
&\frac{1}{k!}\bigg(2\log\bigg(\frac{t+\sqrt{t^{2}+4}}{2}\bigg)\bigg)^{k}=\sum_{n=k}^{\infty}T_{1}(n,k)\frac{t^{n}}{n!},\quad  x^{[n]}=\sum_{k=0}^{n}T_{1}(n,k)x^{k}, \\
&\frac{1}{k!}\big(e^{\frac{t}{2}}-e^{-\frac{t}{2}}\big)^{k}=\sum_{n=k}^{\infty}T_{2}(n,k)\frac{t^{n}}{n!}, \quad x^{n}=\sum_{k=0}^{n}T_{2}(n,k)x^{[k]}. \\
\end{align*}
Let $\mathbf{P}=\left\{x^{[n]}\right\}$ be the sequence of central factorials. We find that
\begin{align*}
&S_{2}(n,k;\mathbf{P})=\sum_{l=k}^{n}T_{1}(n,l)S_{2}(l,k),\quad S_{1}(n,k;\mathbf{P})=\sum_{l=k}^{n}S_{1}(n,l)T_{2}(l,k), \\
&\log_{\mathbf{P}}(1+t)=\sum_{n=1}^{\infty}\sum_{l=1}^{n}S_{1}(n,l)T_{2}(l,1)\frac{t^{n}}{n!}=(1+t)^{\frac{1}{2}}-(1+t)^{-\frac{1}{2}}.
\end{align*}
(g) The central Bell polynomials $\mathrm{Bel}_{n}^{(c)}(x)$ are defined by (see \cite{16})
\begin{equation*}
e^{x(e^{\frac{t}{2}}-e^{-\frac{t}{2}})}=\sum_{n=0}^{\infty}\mathrm{Bel}_{n}^{(c)}(x)\frac{t^{n}}{n!}.
\end{equation*}
Then we have
\begin{align*}
\mathrm{Bel}_{n}^{(c)}(x) \sim \Big(1, f(t)=2\log\bigg(\frac{t+\sqrt{t^{2}+4}}{2}\bigg)\Big), \quad \mathrm{Bel}_{n}^{(c)}(x)=\sum_{k=0}^{n}T_{2}(n,k) x^{k}.
\end{align*}
Let $\mathbf{P}=\left\{\mathrm{Bel}_{n}^{(c)}(x)\right\}$ be the sequence of central Bell polynomials. We obtain that
\begin{align*}
&S_{2}(n,k;\mathbf{P})=\sum_{l=k}^{n}T_{2}(n,l)S_{2}(l,k), \quad
S_{1}(n,k;\mathbf{P})=\sum_{l=k}^{n}S_{1}(n,l)T_{1}(l,k), \\
& \log_{\mathbf{P}}(1+t)=\sum_{n=1}^{\infty}\sum_{l=1}^{n}S_{1}(n,l)T_{1}(l,1)\frac{t^{n}}{n!}=2\log\bigg(\frac{\log(1+t)+\sqrt{(\log(1+t))^{2}+4}}{2}\bigg).
\end{align*}
(h) The degenerate central factorial numbers of the first kind, $T_{1,\lambda}(n,k)$, and of the second kind, $T_{2,\lambda}(n,k)$, are respectively defined by (see \cite{17})
\begin{align*}
&\frac{1}{k!}\bigg(\log_{\lambda}\bigg(\frac{t+\sqrt{t^{2}+4}}{2}\bigg)^{2}\bigg)^{k}=\sum_{n=k}^{\infty}T_{1,\lambda}(n,k)\frac{t^{n}}{n!}, \quad x^{[n]}=\sum_{k=0}^{n}T_{1,\lambda}(n,k)(x)_{k, \lambda},\\
&\frac{1}{k!}\big(e_{\lambda}^{\frac{1}{2}}(t)-e_{\lambda}^{-\frac{1}{2}}(t)\big)^{k}=\sum_{n=k}^{\infty}T_{2,\lambda}(n,k)\frac{t^{n}}{n!},\quad (x)_{n,\lambda}=\sum_{k=0}^{n}T_{2,\lambda}(n,k)x^{[k]}.
\end{align*}
The degenerate central Bell polynomials $\mathrm{Bel}_{n,\lambda}^{(c)}(x)$ are defined by (see \cite{17})
\begin{equation*}
e^{x(e_{\lambda}^{\frac{1}{2}}(t)-e_{\lambda}^{-\frac{1}{2}}(t))}=\sum_{n=0}^{\infty}\mathrm{Bel}_{n,\lambda}^{(c)}(x)\frac{t^{n}}{n!}.
\end{equation*}
Then we see that
\begin{align*}
\mathrm{Bel}_{n,\lambda}^{(c)}(x) \sim \bigg(1, f(t)=\log_{\lambda}\bigg(\frac{t+\sqrt{t^{2}+4}}{2}\bigg)^{2}\bigg),\quad
\mathrm{Bel}_{n,\lambda}^{(c)}(x)=\sum_{k=0}^{n}T_{2,\lambda}(n,k) x^{k}.
\end{align*}
Let $\mathbf{P}=\left\{\mathrm{Bel}_{n,\lambda}^{(c)}(x)\right\}$ be the sequence of degenerate central Bell polynomials. We observe that
\begin{align*}
&S_{2}(n,k;\mathbf{P})=\sum_{l=k}^{n}T_{2,\lambda}(n,l)S_{2}(l,k), \quad S_{1}(n,k;\mathbf{P})=\sum_{l=k}^{n}S_{1}(n,l)T_{1,\lambda}(l,k), \\
& \log_{\mathbf{P}}(1+t)=\sum_{n=1}^{\infty}\sum_{l=1}^{n}S_{1}(n,l)T_{1,\lambda}(l,1)\frac{t^{n}}{n!}=\log_{\lambda}\bigg(\frac{\log(1+t)+\sqrt{(\log(1+t))^{2}+4}}{2}\bigg)^{2}.
\end{align*}
(i) Let $\mathbf{P}=\left\{LB_{n}(x)\right\}$ be the sequence of Lah-Bell polynomials. Then it is given by (see \cite{11})
\begin{equation*}
e^{x(\frac{t}{1-t})}=\sum_{n=0}^{\infty}LB_{n}(x)\frac{t^n}{n!}, \quad LB_{n}(x)=\sum_{k=0}^{n}L(n,k)x^{k},
\end{equation*}
so that $LB_{n}(x) \sim (1, f(t)=\frac{t}{1+t})$. We get that
\begin{align*}
&S_{2}(n,k;\mathbf{P})=\sum_{l=k}^{n}L(n,l)S_{2}(l,k), \quad S_{1}(n,k;\mathbf{P})=\sum_{l=k}^{n}(-1)^{l-k}S_{1}(n,l)L(l,k), \\
&\log_{\mathbf{P}}(1+t)=\sum_{n=1}^{\infty}\sum_{l=1}^{n}(-1)^{l-1}S_{1}(n,l)L(l,1)\frac{t^{n}}{n!}=\frac{\log(1+t)}{1+\log(1+t)}.
\end{align*}
(j) Let $\mathbf{P}=\left\{LB_{n,\lambda}(x)\right\}$ be the sequence of degenerate Lah-Bell polynomials. Then it is given by (see \cite{13})
\begin{equation*}
e_{\lambda}^{x}(\frac{t}{1-t})=\sum_{n=0}^{\infty}LB_{n,\lambda}(x)\frac{t^n}{n!}, \quad LB_{n,\lambda}(x)=\sum_{k=0}^{n}L(n,k)(x)_{k,\lambda}.
\end{equation*}
Then we see that $LB_{n,\lambda}(x) \sim (1,\frac{e^{\lambda t}-1}{\lambda+e^{\lambda t}-1})$. We have that
\begin{align*} 
&S_{2}(n,k;\mathbf{P})=\sum_{l=k}^{n}L(n,l)L_{-\lambda}(l,k), \quad S_{1}(n,k;\mathbf{P})=\sum_{l=k}^{n}(-1)^{l-k}S_{1,\lambda}(n,l)L(l,k), \\
&\log_{\mathbf{P}}(1+t)=\sum_{n=1}^{\infty}\sum_{l=1}^{n}(-1)^{l-1}S_{1,\lambda}(n,l)L(l,1)\frac{t^{n}}{n!}=\frac{\log_{\lambda}(1+t)}{1+\log_{\lambda}(1+t)}.
\end{align*}
(k)  Let $\mathbf{P}=\left\{\mathrm{Bel}_{n}(x)\right\}$ be the sequence of Bell polynomials. Then it given by (see \cite{5,24,25})
\begin{equation*}
e^{x(e^{t}-1)}=\sum_{n=0}^{\infty}\mathrm{Bel}_{n}(x)\frac{t^{n}}{n!},\quad \mathrm{Bel}_{n}(x)=\sum_{k=0}^{\infty}S_{2}(n,k)x^{k}.  
\end{equation*}
Then we see that $\mathrm{Bel}_{n}(x) \sim (1, \log (1+t))$. We show that
\begin{align*}
&S_{2}(n,k;\mathbf{P})=\sum_{l=k}^{n}S_2(n,l)S_{2}(l,k), \quad S_{1}(n,k;\mathbf{P})=\sum_{l=k}^{n}S_{1}(n,l)S_{1}(l,k), \\
&\log_{\mathbf{P}}(1+t)=\sum_{n=1}^{\infty} \sum_{l=1}^{n}S_{1}(n,l)S_{1}(l,1)\frac{t^{n}}{n!}=\log \big(1+\log(1+t)\big).
\end{align*}
(l) Let $\mathbf{P}=\left\{\mathrm{Bel}_{n,\lambda}(x)\right\}$ be the sequence of partially degenerate Bell polynomials. Then it given by (see \cite{18})
\begin{equation*}
e^{x(e_{\lambda}(t)-1)}=\sum_{n=0}^{\infty}\mathrm{Bel}_{n,\lambda}(x)\frac{t^{n}}{n!}, \quad \mathrm{Bel}_{n,\lambda}(x)=\sum_{k=0}^{n}S_{2,\lambda}(n,k)x^{k}.
\end{equation*}
Then $\mathrm{Bel}_{n,\lambda}(x) \sim (1, \log_{\lambda}(1+t))$. We derive that
\begin{align*}
&S_{2}(n,k;\mathbf{P})=\sum_{l=k}^{n}S_{2,\lambda}(n,l)S_{2}(l,k), \quad S_{1}(n,k;\mathbf{P})=\sum_{l=k}^{n}S_{1}(n,l)S_{1,\lambda}(l,k), \\
&\log_{\mathbf{P}}(1+t)=\sum_{n=1}^{\infty} \sum_{l=1}^{n}S_{1}(n,l)S_{1,\lambda}(l,1)\frac{t^{n}}{n!}=\log_{\lambda}\big(1+\log(1+t)\big).
\end{align*}
(m) Let $\mathbf{P}=\left\{\phi_{n,\lambda}(x)\right\}$ be the sequence of fully degenerate Bell polynomials. Then it is given by (see \cite{19})
\begin{equation*}
e_{\lambda}^{x}(e_{\lambda}(t)-1)=\sum_{n=0}^{\infty}\phi_{n,\lambda}(x)\frac{t^{n}}{n!}, \quad \phi_{n,\lambda}(x)=\sum_{k=0}^{n}S_{2,\lambda}(n,k)(x)_{k,\lambda}.
\end{equation*}
Then $\phi_{n,\lambda}(x) \sim \big(1, \log_{\lambda}(1+\frac{1}{\lambda}(e^{\lambda t}-1))\big)$. We find that
\begin{align*}
&S_{2}(n,k;\mathbf{P})=\sum_{m=k}^{n}S_{2,\lambda}(n,m)L_{-\lambda}(m,k), \quad S_{1}(n,k;\mathbf{P})=\sum_{l=k}^{n}S_{1,\lambda}(n,l)S_{1,\lambda}(l,k), \\
&\log_{\mathbf{P}}(1+t)=\sum_{n=1}^{\infty} \sum_{l=1}^{n}S_{1,\lambda}(n,l)S_{1,\lambda}(l,1)\frac{t^{n}}{n!}=\log_{\lambda}\big(1+\log_{\lambda}(1+t)\big).
\end{align*}
(n) Let $\mathbf{P}=\left\{M_{n}(x)\right\}$ be the sequence of Mittag-Leffler polynomials. That is, $M_{n}(x) \sim (1,f(t)=\frac{e^{t}-1}{e^{t}+1})$, with $\bar{f}(t)=\log \big(\frac{1+t}{1-t}\big)$ (see \cite{25}). We obtain that
\begin{align*}
&S_{2}(n,k;\mathbf{P})=2^{k}L(n,k), \quad S_{1}(n,k;\mathbf{P})=(-1)^{n-k}L(n,k)\frac{1}{2^{n}}, \\
&\log_{\mathbf{P}}(1+t)=\sum_{n=1}^{\infty}(-1)^{n-1}L(n,1)\frac{1}{2^{n}}\frac{t^{n}}{n!}=\frac{t}{2+t}.
\end{align*}
(o) Let $\mathbf{P}=\left\{L_{n}(x)\right\}$ be the sequence of Laguerre polynomials of order -1 (see \cite{25}).
That is, $L_{n}(x) \sim (1, f(t)=\frac{t}{t-1})$, with $\bar{f}(t)=\frac{t}{t-1}$. 
We observe that
\begin{align*}
&S_{2}(n,k;\mathbf{P})=\sum_{l=k}^{n}(-1)^{l}L(n,l)S_{2}(l,k), \quad
S_{1}(n,k;\mathbf{P})=(-1)^{k}\sum_{l=k}^{n}S_{1}(n,l)L(l,k), \\
&\log_{\mathbf{P}}(1+t)=-\sum_{n=1}^{\infty}\sum_{l=1}^{n}S_{1}(n,l)L(l,1)\frac{t^{n}}{n!}=-\frac{\log(1+t)}{1-\log(1+t)}.
\end{align*}

\section{Conclusion}
In this paper, we have developed a unified approach to Stirling numbers associated with a delta series $f(t)$, which corresponds to Stirling numbers associated with the sequence $p_{n}(x) \sim (1,f(t))$.  We have successfully addressed the structural deficiencies found in recent probabilistic Stirling and $B$-Stirling number variants. Specifically, our redefined Stirling numbers $S_{1}(n,k; f(t))$ and $S_{2}(n,k; f(t))$ restore the fundamental orthogonality and inverse relations, ensuring that these sequences remain robust tools for inversion formulas and combinatorial identities. We established the Schl\"{o}milch-type formula for these generalized numbers. This formula not only provides a direct computational link between the two kinds of Stirling numbers but also allows for a novel expansion of the associated logarithm $\log_{f(t)}(1+t)$. The fifteen examples detailed in Section 3 demonstrate that our framework is not merely a theoretical exercise but a versatile system capable of recovering and extending a wide array of known results. Whether applied to degenerate cases or probabilistic extensions, the delta series approach provides a consistent and computationally efficient methodology. Future research may involve applying this approach to other special sequences or extending these results to the $q$-analogues.

\end{document}